\documentclass[12pt]{amsart}
\usepackage[headings]{fullpage}
\usepackage{amssymb,epic,eepic,epsfig,amsbsy,amsmath}
\usepackage[all]{xy}
\numberwithin{equation}{section}
                        \theoremstyle{plain}
\newcommand{\psdraw}[2]
         {\begin{array}{c} \hspace{-1.3mm}
         \raisebox{-4pt}{\psfig{figure=#1.eps,width=#2}}
         \hspace{-1.9mm}\end{array}}
         
\numberwithin{equation}{section}

\newtheorem{theorem}{Theorem}[section]
\newtheorem{lemma}[theorem]{Lemma}
\newtheorem{proposition}[theorem]{Proposition}

\newtheorem{thm}{Theorem}

\theoremstyle{definition}
\newtheorem{remark}{Remark}

\def\BC{\mathbb C}

\DeclareMathOperator{\tr}{\mathrm tr}

\def\la{\langle}
\def\ra{\rangle}

\begin{document}

\title[The universal character ring of some families of one-relator groups]{The universal character ring of some families of one-relator groups}

\author[Anh T. Tran]{Anh T. Tran}
\address{Department of Mathematics, The Ohio State University, Columbus, OH 43210, USA}
\email{tran.350@osu.edu}

\thanks{2000 {\em Mathematics Classification:} 57M27.}

\begin{abstract}
We study the universal character ring of some families of one-relator groups. As an  application, we calculate the universal character ring of two-generator one-relator groups whose relators are palindromic, and, in particular, of the $(-2,2m+1,2n+1)$-pretzel knot for all integers $m$ and $n$. For the $(-2,3,2n+1)$-pretzel knot, we give a simple proof of a result in \cite{LTaj} on its universal character ring, and an elementary proof of a result in \cite{Ma} on the number of irreducible components of its character variety.
\end{abstract}

\maketitle

\setcounter{section}{-1}

\section{Introduction}

\subsection{The character variety and the universal character ring} The set of representations of a finitely presented group $G$ into $SL_2(\BC)$ is an algebraic set defined over $\BC$, on which
$SL_2(\BC)$ acts by conjugation. The set-theoretic
quotient of the representation space by that action does not
have good topological properties, because two representations with
the same character may belong to different orbits of that action. A better
quotient, the algebro-geometric quotient denoted by $X(G)$
(see \cite{CS, LM}), has the structure of an algebraic
set. There is a bijection between $X(G)$ and the set of all
characters of representations of $G$ into $SL_2(\BC)$, hence
$X(G)$ is usually called the {\em character variety} of $G$. It is determined by the traces of some fixed elements $g_1, \cdots, g_k$ in $G$. More precisely, one can find $g_1, \cdots, g_k$ in $G$ such that for every element $g$ in $G$ there exists a polynomial $P_g$ in $k$ variables such that for any representation $\rho: G \to SL_2(\BC)$ one has $\tr(\rho(g)) = P_g(x_1, \cdots, x_k)$ where $x_j:=\tr(\rho(g_j))$. The {\em universal character ring} of $G$ is  then defined to be the quotient of the polynomial ring $\BC[x_1, \cdots, x_k]$ by the ideal generated by all expressions of the form $\tr(\rho(u))-\tr(\rho(v))$, where $u$ and $v$ are any two words in $g_1, \cdots, g_k$ which are equal in $G$, c.f. \cite{LTaj}. The universal character ring of $G$ is actually independent of the choice of $g_1, \cdots, g_k$. The quotient of the universal character ring of $G$ by its nil-radical is equal to the ring of regular functions on the character variety $X(G)$.

\subsection{Main results} 

Let $F_{a,w}:=\la a,w \ra$ be the free group in 2 letters $a$ and $w$. The character variety of $F_{a,w}$ is isomorphic to $\BC^3$ by the Fricke-Klein-Vogt theorem, see e.g. \cite{LM}. For every word $u$ in $F_{a,w}$ there is a \emph{unique} polynomial $P_u$ in 3 variables such that for any representation $\rho: F_{a,w} \to SL_2(\BC)$ one has $\tr (\rho(u))=P_u (x,y,z)$ where $x:=\tr(\rho(a)),~y:=\tr(\rho(w))$ and $z:=\tr(\rho(aw))$. Thus for every representation $\rho: G \to SL_2(\BC)$, where $G$ is a group generated by $a$ and $w$, we consider $x,y,$ and $z$ as functions of $\rho$. 

For a word $u$ in $F_{a,w}$, we denote by $\overleftarrow{u}$ the word obtained from $u$ by writing the letters in $u$ in reversed order. The word $u$ is called a \emph{palindrome} if $\overleftarrow{u}=u$. 

In this paper we calculate the universal character ring of some families of two-generator one-relator groups as follows.

\begin{thm}
The universal character ring of the group $\la a,w \mid w^n \overleftarrow{r}=r^{-1}w^{n-1}\ra$ is the quotient of the polynomial ring $\BC[x,y,z]$ by the ideal generated by the two polynomials $P_{\overleftarrow{r}}-P_{r^{-1}w^{-1}}$ and $P_{w^n \overleftarrow{r}a}-P_{r^{-1}w^{n-1}a}.$
\label{main1}
\end{thm}

\begin{thm}
The universal character ring of the group $\la a,w \mid w^n \overleftarrow{r}=r^{-1}w^{n-2}\ra$ is the quotient of the polynomial ring $\BC[x,y,z]$ by the ideal generated by the two polynomials $P_{\overleftarrow{r}}-P_{r^{-1}w^{-2}}$ and $P_{w^n \overleftarrow{r}aw^{-1}}-P_{r^{-1}w^{n-2}aw^{-1}}.$
\label{main2}
\end{thm}

As an application of Theorem \ref{main1}, we immediately obtain a simple proof of the following result in \cite{LTaj} on the universal character ring of the $(-2,3,2n+1)$-pretzel knot.

\begin{thm} [\cite{LTaj}] The fundamental group of the $(-2,3,2n+1)$-pretzel knot is isomorphic to the group $\la a,w \mid w^n \overleftarrow{r}=r^{-1}w^{n-1}\ra$ where $r:=a^{-1}w^{-1}a^{-1}wa.$ Hence its universal character ring is the quotient of the polynomial ring $\BC[x,y,z]$ by the ideal generated by the two polynomials $Q:=P_{\overleftarrow{r}}-P_{r^{-1}w^{-1}}$ and $R_n:=P_{w^n \overleftarrow{r}a}-P_{r^{-1}w^{n-1}a}.$ Explicitly,
\begin{eqnarray*}
Q &=& x - x y + (-3 + x^2 + y^2) z - x y z^2 + z^3,\\
R_n &=& S_{n-2}(y)+S_{n-3}(y)-S_{n-4}(y)-S_{n-5}(y)-S_{n-2}(y) \, x^2 \\
   && + \, \big( S_{n-1}(y)+S_{n-3}(y)+S_{n-4}(y) \big) \, xz-\big( S_{n-2}(y)+S_{n-3}(y) \big)\,z^2, \nonumber
\end{eqnarray*}
where $S_k(y)$'s are the Chebychev polynomials defined by $S_0(y)=1,~S_1(y)=y$ and $S_{k+1}(y)=yS_{k}(y)-S_{k-1}(y)$ for all integers $k$.
\label{aj}
\end{thm}

Applying Theorem \ref{aj}, we also give an elementary proof of the following result in \cite{Ma} on the character variety of the $(-2,3,2n+1)$-pretzel knot.

\begin{thm} [\cite{Ma}]
Suppose $n \not= 0,\,1,\,2$. Then character variety of the hyperbolic $(-2,3,2n+1)$-pretzel knot has 2 irreducible components if $2n+1$ is not divisible by $3$, and has 3 irreducible components if $2n+1$ is divisible by $3$.
\label{character}
\end{thm}

As another application of Theorems \ref{main1} and \ref{main2}, we calculate the universal character ring of the group $G=\la a,w \mid R=1\ra$ where $R$ is a palindromic word in $F_{a,w}$. Since $R$ is palindromic, it either has the form $R=\overleftarrow{r}gr$ or $R=\overleftarrow{r}g^2r$, where $r$ is a word in $F_{a,w}$ and $g$ is either $a,\, a^{-1},\, w$, or $w^{-1}$. Without loss of generality, we consider the case $g=w$ only.

By setting $n=0$ in Theorems \ref{main1} and \ref{main2}, we obtain

\begin{thm}
The universal character ring of the group $\la a,w \mid \overleftarrow{r}wr=1\ra$, where $r$ is a word in $a$ and $w$, is the quotient of the polynomial ring $\BC[x,y,z]$ by the ideal generated by the two polynomials $P_{\overleftarrow{r}}-P_{r^{-1}w^{-1}}$ and $P_{\overleftarrow{r}a}-P_{r^{-1}w^{-1}a}.$
\label{odd}
\end{thm}

\begin{thm}
The universal character ring of the group $\la a,w \mid \overleftarrow{r}w^2r=1\ra$, where $r$ is a word in $a$ and $w$, is the quotient of the polynomial ring $\BC[x,y,z]$ by the ideal generated by the two polynomials $P_{\overleftarrow{r}}-P_{r^{-1}w^{-2}}$ and $P_{\overleftarrow{r}aw^{-1}}-P_{r^{-1}w^{-2}aw^{-1}}.$
\label{even}
\end{thm}

\begin{remark}
By \cite{HTT} tunnel number one knots have presentations with two generators and one relator, where the relator is palindromic in the two generators.  Hence Theorems \ref{odd} and \ref{even} can be applied to calculate the universal character ring of the knot group of tunnel number one knots. 
\end{remark}

In our joint work with T. Le on the AJ conjecture of \cite{Ga, Ge, FGL} which relates the A-polynomial and the colored Jones polynomials of a knot, it is important to understand the universal character
ring of the knot group \cite{Le06, LTaj}. The universal character ring has been so far calculated for a few link groups, including two-bridge knot groups \cite{Le93, PS}, the $(-2,3,2n+1)$-pretzel knot groups \cite{LTaj} (see also Theorem \ref{aj} above), two-bridge link groups \cite{LTskein}, and the $(-2,2m+1,2n)$-pretzel link groups \cite{pretzel-link}.

In the present paper we consider the $(-2,2m+1,2n+1)$-pretzel knot group, where $m$ and $n$ are integers. As an application of Theorem \ref{odd} we will show the following

\begin{thm}
The fundamental group of the $(-2,2m+1,2n+1)$-pretzel knot is isomorphic to the group $\la a,w \mid \overleftarrow{r}wr=1\ra$ where
\begin{eqnarray*}
r &=& \begin{cases} su^{k-1}awaw^{-1}a^{-1}u^{-k}\quad & \text{if} \quad n=2k,\\
su^kawa^{-1}w^{-1}a^{-1}u^{-k}
\quad & \text{if} \quad n=2k+1.
\end{cases}\\
u &=& (awaw^{-1})^{1-m}w, \quad \text{and}\\
s &=& \begin{cases} a(w^{-1}awa)^{-l} & \text{if } m=2l,\\
(w^{-1}awa)^{-l} \quad & \text{if } m=2l+1.
             \end{cases}.
\end{eqnarray*}
Hence its universal character ring is the quotient of the polynomial ring $\BC[x,y,z]$ by the ideal generated by the two polynomials $P_{\overleftarrow{r}}-P_{r^{-1}w^{-1}}$ and $P_{\overleftarrow{r}a}-P_{r^{-1}w^{-1}a}.$
\label{pretzelknot}
\end{thm}

\subsection{Plan of the paper} In Section 1, we collect preliminary facts and lemmas that will be repeated used in the proofs of the main theorems of the paper. In Section 2, we consider the universal character ring of groups and prove Theorems \ref{main1} and \ref{main2}. In Section 3, we study the universal character ring of pretzel knots and prove Theorems \ref{character} and \ref{pretzelknot}.

\subsection{Acknowledgement} The author would like to thank Thang T.Q. Le for helpful discussions. He wishes to thank the referee for comments and suggestions that greatly improves the presentation of the paper.

\section{Preliminary facts and lemmas}

\subsection{The backward operator} Recall from the Introduction that for a word $u$ in $F_{a,w}$, we denote by $\overleftarrow{u}$ the word obtained from $u$ by writing the letters in $u$ in reversed order.
\begin{lemma}
One has $\overleftarrow{\overleftarrow{u}}=u$ , $\overleftarrow{uv}=\overleftarrow{v}\overleftarrow{u}$ and $\overleftarrow{u^{-1}}=\overleftarrow{u}^{-1}$ for all words $u,v$ in $F_{a,w}$. Hence $\overleftarrow{u^n}=\overleftarrow{u}^n$ for all integers $n$.
\label{def}
\end{lemma}

\begin{proof}
The first two identities follow directly from the definition of the backward operator $\overleftarrow{\cdot}$. The third identity follows from the second one by taking $v=u^{-1}$.
\end{proof}

We will also use the following result in \cite{Le93, pretzel-link}.

\begin{lemma} 
One has $P_{uv} = P_{\overleftarrow{u}\overleftarrow{v}}$ for all words $u,v$ in $F_{a,w}$.
\label{zero}
\end{lemma}

\subsection{Trace identities} For all matrices $A,B,C$ in $SL_2(\BC)$, the following trace identities are well-known: 
\begin{eqnarray}
\tr A &=& \tr A^{-1}, \label{tr1}\\
\tr {AB} &=& \tr {BA} \label{tr2},\\
\tr {BA}+ \tr{BA^{-1}} &=& (\tr A) (\tr {B}) \label{tr3},\\
\tr {BAC}+ \tr{BA^{-1}C} &=& (\tr A) (\tr {BC}).
\label{Cayley}
\end{eqnarray}
Note that Identities \eqref{tr3} and \eqref{Cayley} follow from the Cayley-Hamilton theorem  $A+A^{-1}=P_A I_{2 \times 2}$, where $I_{2 \times 2}$ is the $2 \times 2$ identity matrix. 

\begin{lemma}
One has $$P_{ucd}+P_{udc}=-P_{cd^{-1}}P_u+P_cP_{ud}+P_dP_{uc}$$
\label{4}
for all words $c,d,u$ in $F_{a,w}$.
\end{lemma}

\begin{proof} 
We have
\begin{align*}
P_{u(dc)} &= P_{dc} P_{u}-P_{uc^{-1}d^{-1}} & \text{by Identity~} \eqref{tr3}\\
&=P_{dc}P_u-(P_{d}P_{uc^{-1}}-P_{uc^{-1}d}) & \text{by Identity~} \eqref{tr3}\\
&=P_{dc}P_u-P_d P_{uc^{-1}}+P_{uc^{-1}d} & \\
&=P_{dc} P_u-P_d(P_cP_{u}-P_{uc})+(P_cP_{ud}-P_{ucd}) & \text{by Identities~} \eqref{tr3} \text{~and~}\eqref{Cayley}\\
&=(P_{cd}-P_cP_d)P_u+P_cP_{ud}+P_dP_{uc}-P_{ucd} & \text{by Identity~} \eqref{tr2}\\
&= -P_{cd^{-1}}P_u+P_cP_{ud}+P_dP_{uc}-P_{ucd} & \text{by Identity~} \eqref{tr3}
\end{align*}
The lemma follows.
\end{proof}

\subsection{Chebyshev polynomials} Let $S_k(t)$'s be the Chebychev polynomials defined by $S_0(t)=1,~S_1(t)=t$ and $S_{k+1}(t)=tS_{k}(t)-S_{k-1}(t)$ for all integers $k$.

It is easy to see that $S_k(2)=k+1$ and $S_k(-2)=(-1)^k(k+1)$ for all integers $k$.

\begin{lemma}
One has $S^2_{k}(t)-tS_{k}(t)S_{k-1}(t)+S_{k-1}^2(t)=1$.
\label{hom}
\end{lemma}

\begin{proof}
Let $g_k(t)=S_k^2(t)-tS_k(t)S_{k-1}(t)+S_{k-1}^2(t)$. Then
\begin{eqnarray*}
g_k(t) &=& (S_k(t)-tS_{k-1}(t))S_{k}(t)+S_{k-1}^2(t)\\
&=& -S_{k-2}(t)S_{k}(t)+(tS_{k-2}-S_{k-3}(t))S_{k-1}(t)\\
&=& S_{k-2}(t)(tS_{k-1}(t)-S_{k}(t))-S_{k-3}(t)S_{k-1}(t)\\
&=& S_{k-2}^2(t)-(tS_{k-2}(t)-S_{k-1}(t))S_{k-1}(t)\\
&=& g_{k-1}(t).
\end{eqnarray*} 
It means that $g_k(t)$ does not depend on $k$ and so $g_k(t)=g_0(t)=1$. Hence $S^2_{k}(t)-tS_{k}(t)S_{k-1}(t)+S_{k-1}^2(t)=g_{k}(t)=1$.
\end{proof}

\section{Proof of Theorems \ref{main1} and \ref{main2}}
\label{Thms12}

\subsection{The universal character ring of two-generator one-relator groups}

\begin{proposition}
Let $G:=\la a,w \mid u=v\ra$, where $u$ and $v$ are two words in $F_{a,w}$. Then the universal character ring of $G$ is the quotient of the polynomial ring $\BC[x,y,z]$ by the ideal generated by the four polynomials $P_u-P_v,~P_{ua}-P_{va},~P_{uw}-P_{vw}$ and $P_{uwa}-P_{vwa}$.
\label{prop}
\end{proposition}

\begin{proof}
By \cite[Prop 1.1]{pretzel-link}, the universal character ring of $G$ is the quotient of the polynomial ring $\BC[x,y,z]$ by the ideal generated by the five polynomials $P_u-P_v,~P_{ua}-P_{va},~P_{uw}-P_{vw},~P_{uaw}-P_{vaw}$ and $P_{uwa}-P_{vwa}.$ From Lemma \ref{4} it follows that
$$(P_{uaw}-P_{vaw})+(P_{uwa}-P_{vwa})=-P_{aw^{-1}}(P_u-P_v)+P_a(P_{uw}-P_{vw})+P_w(P_{ua}-P_{va}).$$
Hence the universal character ring of $G$ is the quotient of the polynomial ring $\BC[x,y,z]$ by the ideal generated by the four polynomials $P_u-P_v,~P_{ua}-P_{va},~P_{uw}-P_{vw}$ and $P_{uwa}-P_{vwa}$. 
\end{proof}

\begin{remark}
From the proof of \cite[Prop 1.1]{pretzel-link}, it is easy to see that the polynomial $P_{uwa}-P_{vwa}$ in Proposition \ref{prop} can be replaced by any polynomial of the form $P_{ug_1^{\varepsilon_1}g_2^{\varepsilon_2}}-P_{vg_1^{\varepsilon_1}g_2^{\varepsilon_2}}$, where $\{g_1,g_2\}=\{a,w\}$ and $\varepsilon_1, \varepsilon_2 \in \{\pm 1\}$.
\label{nx}
\end{remark}

\subsection{Proof of Theorem \ref{main1}} The group in Theorem \ref{main1} is $\la a,w \mid w^n\overleftarrow{r}=r^{-1}w^{n-1}\ra.$ 

To prove Theorem 1 we will need the following propositions.

\begin{proposition}
One has
$$P_{w^n\overleftarrow{r}\overleftarrow{u}}-P_{r^{-1}w^{n-1}\overleftarrow{u}}=(P_{w^n\overleftarrow{r}uw^{-1}}-P_{r^{-1}w^{n-1}uw^{-1}})-P_{uw^{n-1}}(P_{\overleftarrow{r}}-P_{r^{-1}w^{-1}}),$$
for all words $u$ in $F_{a,w}$.
\label{main1prop}
\end{proposition}

\begin{proof}
We have
\begin{align*}
P_{w^n\overleftarrow{r}\overleftarrow{u}} &= P_{\overleftarrow{w^n}\overleftarrow{\overleftarrow{r}\overleftarrow{u}}} & \text{by Lemma~}\ref{zero}\\
&= P_{w^nur} & \text{by Lemma~}\ref{def}\\
&= P_{(w^{n-1}u)(rw)} \\
&= P_{w^{n-1}u} P_{rw} -P_{(w^{n-1}u)(rw)^{-1}}& \text{by Identity~}\eqref{tr3}\\
&= P_{uw^{n-1}} P_{r^{-1}w^{-1}}-P_{r^{-1}w^{n-1}uw^{-1}} & \text{by Identities~}\eqref{tr1} \text{~and~}\eqref{tr2}
\end{align*}
Similarly,
\begin{align*}
P_{r^{-1}w^{n-1}\overleftarrow{u}} &= P_{\overleftarrow{r^{-1}}\overleftarrow{w^{n-1}\overleftarrow{u}}} & \text{by Lemma~}\ref{zero}\\
&= P_{\overleftarrow{r}^{-1} uw^{n-1}} & \text{by Lemma~}\ref{def}\\
&= P_{uw^{n-1}}P_{\overleftarrow{r}}-P_{\overleftarrow{r}uw^{n-1}} & \text{by Identity~}\eqref{tr3}\\
&= P_{uw^{n-1}}P_{\overleftarrow{r}}-P_{w^n\overleftarrow{r}uw^{-1}} & \text{by Identities~}\eqref{tr1} \text{~and~}\eqref{tr2}
\end{align*}
Hence $P_{w^n\overleftarrow{r}\overleftarrow{u}}-P_{r^{-1}w^{n-1}\overleftarrow{u}}=(P_{w^n\overleftarrow{r}uw^{-1}}-P_{r^{-1}w^{n-1}uw^{-1}})-P_{uw^{n-1}}(P_{\overleftarrow{r}}-P_{r^{-1}w^{-1}}).$
\end{proof}

\begin{proposition}
One has $$P_{w^n\overleftarrow{r}}-P_{r^{-1}w^{n-1}}=-(S_{n-1}(y)+S_{n-2}(y))(P_{\overleftarrow{r}}-P_{r^{-1}w^{-1}}).$$
\label{1}
\end{proposition}

\begin{proof}

Let $g_n=P_{w^n\overleftarrow{r}}-P_{r^{-1}w^{n-1}}$. By applying Identity \eqref{tr3}, it is easy to show that $g_{n+1}=yg_n-g_{n-1}$ for all integers $n$ (note that $P_w=y$). By definition, $g_0=P_{\overleftarrow{r}}-P_{r^{-1}w^{-1}}$. Applying Lemmas \ref{def}, \ref{zero} and Identity \eqref{tr1}, we get 
$$g_1=P_{w\overleftarrow{r}}-P_{r^{-1}}=P_{wr}-P_{r}=P_{r^{-1}w^{-1}}-P_{\overleftarrow{r}}=-(P_{\overleftarrow{r}}-P_{r^{-1}w^{-1}}).$$
Hence, by induction on $n$, we can easily show that $g_n=-(S_{n-1}(y)+S_{n-2}(y))(P_{\overleftarrow{r}}-P_{r^{-1}w^{-1}})$. Proposition \ref{1} follows.
\end{proof}

We now prove Theorem \ref{main1}. Let $f_n(u)=P_{w^n\overleftarrow{r}u}-P_{r^{-1}w^{n-1}u}$ for $u \in F_{a,w}$. Then, by Proposition \ref{prop}, the universal character ring of the group $G=\la a,w \mid w^n \overleftarrow{r}=r^{-1}w^{n-1} \ra$ is the quotient of the polynomial ring $\BC[x,y,z]$ by the ideal $I$ generated by the four polynomials $f_n(1),~f_n(a),~f_n(w)$ and $f_n(wa)$, where $x=P_a,~y=P_w$ and $z=P_{aw}$.

Let $Q=P_{\overleftarrow{r}}-P_{r^{-1}w^{-1}}$. Since $\overleftarrow{r}$ and $r^{-1}w^{-1}$ are conjugate in $G$ (by $w^n$), it is clear from the definition of the universal character ring that $Q$ is contained in the ideal $I$.

By Proposition \ref{main1prop},
$$f_n(\overleftarrow{u})=f_n(uw^{-1})-P_{uw^{n-1}}Q.$$
In particular, we have $f_n(w)=f_n(1)-P_{w^n}Q$ and $f_n(wa)=f_n(a)-P_{aw^n}Q.$ By Proposition \ref{1}, $f_n(1)=-(S_{n-1}(y)+S_{n-2}(y))Q$. Hence the ideal $I$ is generated by the two polynomials $Q=P_{\overleftarrow{r}}-P_{r^{-1}w^{-1}}$ and $f_n(a)=P_{w^n \overleftarrow{r}a}-P_{r^{-1}w^{n-1}a}$. Theorem \ref{main1} follows.

\subsection{Proof of Theorem \ref{main2}} The group in Theorem \ref{main2} is $\la a,w \mid w^n\overleftarrow{r}=r^{-1}w^{n-2}\ra.$

To prove Theorem 2 we will need the following propositions.

\begin{proposition}
One has
$$P_{w^n\overleftarrow{r}\overleftarrow{u}}-P_{r^{-1}w^{n-2}\overleftarrow{u}}=P_{w^n\overleftarrow{r}(wuw^{-1})}-P_{r^{-1}w^{n-2}(wuw^{-1})}$$
for all words $u$ in $F_{a,w}$.
\label{key}
\end{proposition}
\begin{proof} From the proof of Proposition \ref{main1prop} we have
\begin{eqnarray*}
P_{w^n\overleftarrow{r}\overleftarrow{u}} &=& P_{rw}P_{uw^{n-1}}-P_{r^{-1}w^{n-2}(wuw^{-1})}.
\end{eqnarray*}
Similarly,
\begin{align*}
P_{r^{-1}w^{n-2}\overleftarrow{u}} &= P_{\overleftarrow{r^{-1}}\overleftarrow{w^{n-2}\overleftarrow{u}}} & \text{by Lemma~}\ref{zero}\\
&= P_{\overleftarrow{r}^{-1}uw^{n-2}} & \text{by Lemma~}\ref{def}\\
&= P_{(\overleftarrow{r}w)^{-1}(uw^{n-1})} & \text{by Identity~}\eqref{tr2}\\
&= P_{\overleftarrow{r}w}P_{uw^{n-1}}-P_{w^n\overleftarrow{r}(wuw^{-1})} & \text{by Identities ~}\eqref{tr1} \text{~and~} \eqref{tr2}
\end{align*}
Hence $$P_{w^n\overleftarrow{r}\overleftarrow{u}}-P_{r^{-1}w^{n-2}\overleftarrow{u}}=P_{w^n\overleftarrow{r}(wuw^{-1})}-P_{r^{-1}w^{n-2}(wuw^{-1})}+P_{uw^{n-1}}(P_{rw}-P_{\overleftarrow{r}w}).$$
The proposition follows, since $P_{rw}-P_{\overleftarrow{r}w}=0$ by Lemma \ref{zero}.
\end{proof}

\begin{proposition}
One has 
$$(P_{w^n\overleftarrow{r}u}-P_{r^{-1}w^{n-2}u})+(P_{w^n\overleftarrow{r}(wuw^{-1})}-P_{r^{-1}w^{n-2}(wuw^{-1})})$$
$$= -P_{uw^{-2}}(P_{w^n\overleftarrow{r}}-P_{r^{-1}w^{n-2}})+P_{uw^{-1}}(P_{w^n\overleftarrow{r}w}-P_{r^{-1}w^{n-2}w})+P_{w}(P_{w^n\overleftarrow{r}(uw^{-1})}-P_{r^{-1}w^{n-2}(uw^{-1})})$$
for all words $u$ in $F_{a,w}$.
\label{them}
\end{proposition}

\begin{proof}
By Lemma \ref{4}, for any word $v$ in $F_{a,w}$, we have
\begin{eqnarray*}
P_{vu}+P_{v(wuw^{-1})} &=& P_{v(uw^{-1})w}+P_{vw(uw^{-1})}\\
&=& -P_{uw^{-2}}P_v+P_{uw^{-1}}P_{vw}+P_{w}P_{v(uw^{-1})}.
\end{eqnarray*}
In particular,
\begin{eqnarray*}
P_{w^n\overleftarrow{r}u}+P_{w^n\overleftarrow{r}(wuw^{-1})} &=& -P_{uw^{-2}}P_{w^n\overleftarrow{r}}+P_{uw^{-1}}P_{w^n\overleftarrow{r}w}+P_{w}P_{w^n\overleftarrow{r}(uw^{-1})},\\
P_{r^{-1}w^{n-2}u}+P_{r^{-1}w^{n-2}(wuw^{-1})} &=& -P_{uw^{-2}}P_{r^{-1}w^{n-2}}+P_{uw^{-1}}P_{r^{-1}w^{n-2}w}+P_{w}P_{r^{-1}w^{n-2}(uw^{-1})}.
\end{eqnarray*}
The proposition follows by taking the difference of the two identities above.
\end{proof}

\begin{proposition}
One has 
\begin{eqnarray*}
P_{w^n\overleftarrow{r}}-P_{r^{-1}w^{n-2}} &=& -S_{n-2}(y)(P_{\overleftarrow{r}}-P_{r^{-1}w^{-2}}),\\
P_{w^n\overleftarrow{r}w}-P_{r^{-1}w^{n-2}w} &=& -S_{n-1}(y)(P_{\overleftarrow{r}}-P_{r^{-1}w^{-2}}).
\end{eqnarray*}
\label{2}
\end{proposition}

\begin{proof}
The proof is similar to that of Proposition \ref{1}, so we omit the details.
\end{proof}

We now prove Theorem 2. Let $f_n(u)=P_{w^n\overleftarrow{r}u}-P_{r^{-1}w^{n-2}u}.$ Then, by Proposition \ref{prop} and Remark \ref{nx}, the universal character ring of the group $G=\la a,w \mid w^n\overleftarrow{r}=r^{-1}w^{n-2} \ra$ is the quotient of the polynomial ring $\BC[x,y,z]$ by the ideal $I$ generated by the four polynomials $f_n(1),~f_n(a),~f_n(w)$ and $f_n(aw^{-1})$, where $x=P_a,~y=P_w$ and $z=P_{aw}$.

By Proposition \ref{key}, $f_n(\overleftarrow{u})=f_n(wuw^{-1})$. Moreover, we have 
$$f_n(u)+f_n(wuw^{-1}) = -P_{uw^{-2}} f_n(1)+P_{uw^{-1}}f_n(w) + P_{w}f_n(uw^{-1})$$
by Proposition \ref{them}. Hence 
$$f_n(u)+f_n(\overleftarrow{u})= -P_{uw^{-2}} f_n(1)+P_{uw^{-1}}f_n(w) + P_{w}f_n(uw^{-1}).$$
In particular, 
\begin{equation}
2f_n(a)=-P_{aw^{-2}}f_n(1)+ P_{aw^{-1}}f_n(w)+ P_w f_n(aw^{-1}).
\label{a}
\end{equation}

Let $Q=P_{\overleftarrow{r}}-P_{r^{-1}w^{-2}}$. Since $\overleftarrow{r}$ and $r^{-1}w^{-2}$ are conjugate in $G$ (by $w^n$), it is clear that from the definition of the universal character ring that $Q$ is contained in the ideal $I$.

By Proposition \ref{2}, $f_n(1)=-S_{n-2}(y)Q$ and $f_n(w)=-S_{n-1}(y)Q$. These identities and Identity \eqref{a} imply that the ideal $I$ is generated by $Q$ and $f_n(aw^{-1})$, and so the universal character ring of $G$ is the quotient of the polynomial ring $\BC[x,y,z]$ by the ideal generated by the two polynomials $P_{\overleftarrow{r}}-P_{r^{-1}w^{-2}}$ and $P_{w^n\overleftarrow{r}aw^{-1}}-P_{r^{-1}w^{n-2}aw^{-1}}.$

\section{Pretzel knots} 

\subsection{Proof of Theorem \ref{pretzelknot}} The fundamental group of the $(-2,2m+1,2n+1)$-pretzel knot is $$
\pi=\la a,b,c \mid bab^{-1}=(ac)^{-m}c(ac)^m,~a^{-1}ba=(cb)^nc(cb)^{-n}\ra.
$$
The first relation in the group $\pi$ is $(ac)^m ba=c(ac)^m b$, i.e. $a(ca)^{m-1}cba=ca(ca)^{m-1}cb.$ Let $w=(ca)^{m-1}cb$ then $awa=caw$. It implies that $ca=awaw^{-1}$ and $cb=(ca)^{1-m}w=(awaw^{-1})^{1-m}w.$ Let $u=(awaw^{-1})^{1-m}w$. Then $cb=u$ and so $$b=c^{-1}u=awa^{-1}w^{-1}a^{-1}(awaw^{-1})^{1-m}w=a(awaw^{-1})^{-m}w.$$ The second relation in the group $\pi$ becomes $(awaw^{-1})^{-m}wa=u^nawaw^{-1}a^{-1}u^{-n},$ which is equivalent to $$(awaw^{-1}a^{-1}u^{-n})^{-1}=(u^{-n}(awaw^{-1})^{-m}wa)^{-1}.$$ Therefore
$$\pi=\la a, w \mid u^{n}awa^{-1}w^{-1}a^{-1}=a^{-1}w^{-1}awau^{n-1} \ra.$$

\begin{figure}[htpb]
$$\psdraw{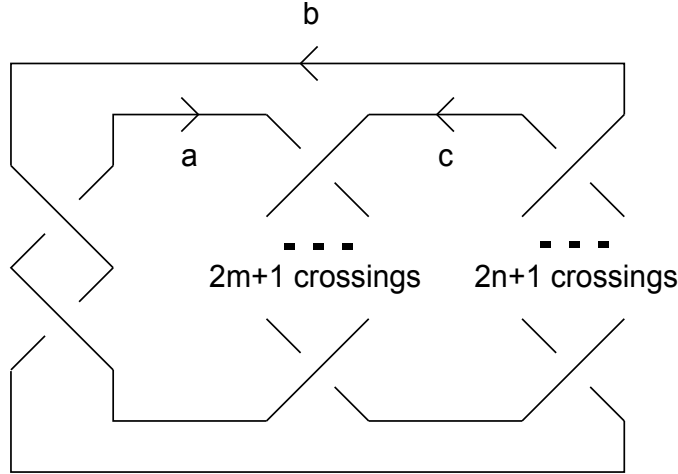}{3.5in}$$
\caption{The $(-2,2m+1,2n+1)$-pretzel knot}
\end{figure}

To proceed, we will need the following lemma.

\begin{lemma} One has $u=\overleftarrow{s}ws$ where
$$s=\begin{cases} a(w^{-1}awa)^{-l} & \text{if } \quad m=2l,\\
(w^{-1}awa)^{-l} \quad & \text{if } \quad m=2l+1.
             \end{cases}$$
In particular, $u$ is palindromic, i.e. $u=\overleftarrow{u}.$
\label{u}
\end{lemma}

\begin{proof} We first note that $(uv)^{k+1}=u(vu)^{k}v$ for all integers $k$. If $m=2l$ then 
\begin{eqnarray*}
u &=& (awaw^{-1})^{-l}(awaw^{-1})^{1-l}w\\
&=&(awaw^{-1})^{-l}[(awa)(w^{-1}awa)^{-l}w^{-1}]w\\
&=&[(awaw^{-1})^{-l}a]w[a(w^{-1}awa)^{-l}].
\end{eqnarray*}
By Lemma \ref{def}, $\overleftarrow{s}=(awaw^{-1})^{-l}a$. Hence $u=\overleftarrow{s}ws$.

If $m=2l+1$ then similarly
\begin{eqnarray*}
u &=& (awaw^{-1})^{-l-1}(awaw^{-1})^{1-l}w\\
&=&(awaw^{-1})^{-l-1}(awa)(w^{-1}awa)^{-l}\\
&=&(awaw^{-1})^{-l}w(w^{-1}awa)^{-l}.
\end{eqnarray*}
Hence $u=\overleftarrow{s}ws$ where $s=(w^{-1}awa)^{-l}$.
\end{proof}

\begin{proposition} One has $\pi=\la a,w \mid \overleftarrow{r}wr=1 \ra$ where 
$$r=\begin{cases} su^{k-1}awaw^{-1}a^{-1}u^{-k}\quad & \text{if} \quad n=2k,\\
su^kawa^{-1}w^{-1}a^{-1}u^{-k}
\quad & \text{if} \quad n=2k+1.
\end{cases}$$
\label{done}
\end{proposition}

\begin{proof}
Recall that $\pi=\la a, w \mid u^{n}awa^{-1}w^{-1}a^{-1}=a^{-1}w^{-1}awau^{n-1} \ra.$ 

If $n=2k$ then the relation in $\pi$ is $u^{2k}awa^{-1}w^{-1}a^{-1}=a^{-1}w^{-1}awau^{2k-1}$, which is equivalent to
$u^{-k}a^{-1}w^{-1}awau^{2k-1}awaw^{-1}a^{-1}u^{-k}=1.$ 

If $n=2k+1$ then the relation in $\pi$ is $u^{2k+1}awa^{-1}w^{-1}a^{-1}=a^{-1}w^{-1}awau^{2k}$, which is equivalent to
$u^{-k}a^{-1}w^{-1}a^{-1}wau^{2k+1}awa^{-1}w^{-1}a^{-1}u^{-k}=1.$  

The proposition then follows from Lemma \ref{u}.
\end{proof}

We now complete the proof of Theorem \ref{pretzelknot}. Proposition \ref{done} and Theorem \ref{odd} imply that the universal character of $\pi$ is the quotient of the polynomial ring $\BC[x,y,z]$ by the ideal generated by the two polynomials $P_{\overleftarrow{r}}-P_{r^{-1}w^{-1}}$ and $P_{\overleftarrow{r}a}-P_{r^{-1}w^{-1}a}.$

\subsection{Proof of Theorem \ref{character} }

Let $V$ be the character variety of the $(-2,3,2n+1)$-pretzel knot. Then by Theorem \ref{aj}, $V$ is the zero locus of the two polynomials $Q$ and $R_n$, where
\begin{eqnarray*} 
Q &=& x-xy+(x^2+y^2-3)z-xyz^2+z^3,\\
R_n &=& (y+2)S_{n-2}(y)-(y^2+y-2)S_{n-3}(y)-S_{n-2}(y)x^2\\
&& + \, ((y-1)S_{n-2}(y)+yS_{n-3}(y))xz-(S_{n-2}(y)+S_{n-3}(y))z^2.
\end{eqnarray*} 

It is known that 3-strand pretzel knots are small knots (see \cite{Oe}), hence by \cite{CS} their character varieties have irreducible components of dimension 1 only. Therefore, all irreducible components of $V$ have dimension exactly 1.

Note that if $n=0,\,1$ or $2$ then the $(-2,3,2n+1)$-pretzel knot is a torus knot, otherwise it is hyperbolic. From now on we suppose that $n \not= 0,\,1,\,2$. 

\begin{lemma}
Suppose $\gcd(2n+1,3)=1$. Then $z \not= 0$ on $V$ except a finite number of points.
\label{z=0}
\end{lemma}

\begin{proof}
Fix $z=0$. Then $Q=x(1-y)$ and $$R_n=(y+2)S_{n-2}(y)-(y^2+y-2)S_{n-3}(y)-S_{n-2}(y)x^2.$$ Note that $S_k(2)=k+1$ for all intergers $k$.

If $x=0$ then $R_n=p(y)$ where $p(y):=(y+2)S_{n-2}(y)-(y^2+y-2)S_{n-3}(y)$. Note that $p(2)=4S_{n-2}(2)-4S_{n-3}(2)=4(n-1)-4(n-2)=4$, hence $p(y)$ is a non-zero polynomial in $y$ and so it has a finite number of roots.

If $y=1$ then $R_n=(3-x^2)S_{n-2}(y)$. Note that $S_{3k+2}(1)=0,~S_{3k}(1)=S_{3k+1}(1)=(-1)^k$. Since $\gcd(2n+1,3)=1$, we have $S_{n-2}(y)=S_{n-2}(1)=\pm 1$. Hence $R_n=0$ if and only if $x=\pm \sqrt{3}$. The lemma follows.
\end{proof}

By Lemma \ref{z=0}, we separate the proof of Theorem \ref{character} into 2 cases: $\gcd(2n+1,3)=1$ and $\gcd(2n+1,3)=3$.

\subsection{The case $\gcd(2n+1,3)=1$} Then by Lemma \ref{z=0}, $z \not= 0$ on $V$ except a finite number of points. Without loss of generality, we may suppose $z \not= 0$ on $V$. Let $Q'=Qz^{-1}$ and $R'_n=R_n+S_{n-2}(y)Q'$. Then we have
\begin{eqnarray*}
Q' = x^2-( yz^2+y-1)z^{-1}x+y^2+z^2-3, \quad \text{and} \quad R'_n = -\alpha z^{-1}x+\beta,
\end{eqnarray*}
where 
\begin{eqnarray*}
\alpha &=& (z^2+y-1)S_{n-2}(y)-yz^2S_{n-3}(y), \\
\beta &=& (y^2+y-1)S_{n-2}(y)-(y^2+y-2+z^2)S_{n-3}(y).
\end{eqnarray*}

To proceed, we will need the following lemma.

\begin{lemma}
One has $\alpha \not=0$ on $V  \cap \{z \not= 0\}$ except a finite number of points.
\label{alpha=0}
\end{lemma}

\begin{proof}
Assume $\alpha =0$ on $V \cap \{z \not= 0\}$. Then $\alpha = \beta =0$, which implies that $(y-1)S_{n-2}(y)= (yS_{n-3}(y)-S_{n-2}(y))z^2$ and $(y^2+y-1)S_{n-2}(y)-(y^2+y-2)S_{n-3}(y)=z^2S_{n-3}(y)$. Hence 
\begin{equation}
(yS_{n-3}(y)-S_{n-2}(y))[(y^2+y-1)S_{n-2}(y)-(y^2+y-2)S_{n-3}(y)]-(y-1)S_{n-2}(y)S_{n-3}(y)=0.
\label{eqn}
\end{equation}

Let $q(y)$ be the left-hand side of the equation \eqref{eqn}. Note that $q(2)=(n-3)(n+3)-(n-1)(n-2)=3n-11 \not= 0$. It implies that $q(y)$ is a non-zero polynomial in $y$ and so it has a finite number of roots. For each root $y$ of $q(y)$, the system $(y-1)S_{n-2}(y)= (yS_{n-3}(y)-S_{n-2}(y))z^2$ and $(y^2+y-1)S_{n-2}(y)-(y^2+y-2)S_{n-3}(y)=z^2S_{n-3}(y)$ has at most 2 solutions $z$, since either $yS_{n-3}(y)-S_{n-2}(y)$ or $S_{n-3}(y)$ is non-zero (by Lemma \ref{hom}). For each solution $(y,z)$ of the system $\alpha=\beta=0$, the equation $Q'=0$ has at most 2 solutions $x$. Therefore $\alpha \not=0$ on $V  \cap \{z \not= 0\}$ except a finite number of points.
\end{proof}

Since $\gcd(2n+1,3)=1$, by Lemmas \ref{z=0} and \ref{alpha=0} we may assume that $\alpha \not=0$ and $z \not= 0$ on $V$. The equation $R'_n=0$ is then equivalent to $x=\frac{z\beta}{\alpha}$. Hence
\begin{eqnarray*}
\alpha^2 Q' &=& z^2\beta ^2-( yz^2+y-1)\beta \alpha+ (y^2+z^2-3)\alpha ^2\\
&=& (-2 + 3 y - y^3 + z^2) \big\{ S_{n-2}(y)^2-(y-1)S_{n-2}(y)S_{n-3}(y)\\
&& - \, (3S_{n-2}(y)^2-(2y+1)S_{n-2}(y)S_{n-3}(y)+2S_{n-3}(y)^2)z^2\\
&& + \, (S_{n-2}(y)^2-yS_{n-2}(y)S_{n-3}(y)+S_{n-3}(y)^2)z^4 \big\}.
\end{eqnarray*}
By Lemma \ref{hom}, we have $S_{n-2}(y)^2-yS_{n-2}(y)S_{n-3}(y)+S_{n-3}(y)^2=1$. It follows that 
\begin{eqnarray*}
\alpha^2 Q'&=& (-2 + 3 y - y^3 + z^2) \big\{ 1+S_{n-2}(y)S_{n-3}(y)-S_{n-3}(y)^2\\
&& - \, (2+S_{n-2}(y)^2-S_{n-2}(y)S_{n-3}(y))z^2+z^4 \big\}\\
&=& (-2 + 3 y - y^3 + z^2)T(y,z)
\end{eqnarray*}
where 
\begin{eqnarray*}
T(y,z) &:=& t_0(y)+t_2(y)z^2+z^4,\\
t_0(y) &:=& 1+S_{n-2}(y)S_{n-3}(y)-S_{n-3}(y)^2,\\
t_2(y) &:=& -(2+S_{n-2}(y)^2-S_{n-2}(y)S_{n-3}(y)).
\end{eqnarray*} 

\begin{lemma}
Suppose $n \not= 1,\,2$. Then $t_0(y) \in \BC[y]$ is a polynomial of positive degree and it does not have any repeated factors. 
\label{E}
\end{lemma}

\begin{proof}
Note that $S_k(2)=k+1$ and $S_k(-2)=(-1)^k(k+1)$ for all integers $k$. It follows that $h_0(y)=1+S_{n-2}(y)S_{n-3}(y)-S_{n-3}(y)^2$ is equal to $n-1$ if $y=2$; and is equal to $(n-1)(5-2n)$ if $y=-2$. Hence $h_0(y) \in \BC[y]$ is a polynomial of positive degree since $n \not= 1,\,2$.

We have $t_0(y)=S_{n-2}(y) \big( S_{n-2}(y)-(y-1)S_{n-3}(y) \big)=S_{n-2}(y)(S_{n-3}(y)-S_{n-4}(y)).$ If $n \ge 4$ then 
$S_{n-2}(y)=\prod_{j=1}^{n-2}(y-2\cos\frac{j\pi}{n-1})$ and $S_{n-3}(y)-S_{n-4}(y)=\prod_{j=1}^{n-3}(y-2\cos\frac{(2j-1)\pi}{2n-5})$ (see, for example, \cite[Lem 4.13]{LTaj}), hence
$$t_0(y)=\prod_{j=1}^{n-2}(y-2\cos\frac{j\pi}{n-1})\prod_{j=1}^{n-3}(y-2\cos\frac{(2j-1)\pi}{2n-5})$$
does not have any repeated factors. 

Similarly, if $n \le -1$ then by letting $n'=-(n+1) \ge 0$, we have
\begin{eqnarray*}
t_0(y) &=& S_{n'+1}(y)(S_{n'+2}(y)-S_{n'+3}(y))\\
&=& -\prod_{j=1}^{n'+1}(y-2\cos\frac{j\pi}{n'+2})\prod_{j=1}^{n'+3}(y-2\cos\frac{(2j-1)\pi}{2n'+7})
\end{eqnarray*}
since $S_{k}(y)=-S_{-k-2}(y)$ for all integers $k$. Hence $h_0(y)$ does not have any repeated factors. If $n=0$ then $t_0(y)=-(y^2-y-1)$. If $n=3$ then $t_0(y)=y$. The lemma follows.
\end{proof}

\begin{proposition}
Suppose $n \not= 0,\,1,\,2$. Then $T(y,z)$ is irreducible in $\BC[y,z]$.
\label{irred}
\end{proposition}

\begin{proof}
If $T(y,z)$ has a factor $z+f(y)$ where $f(y) \not\equiv 0$, then $z-f(y)$ is also a factor of $T(y,z)$. Hence $T(y,z)$ has a factor $z^2-f^2(y)$. 

If $T(y,z)$ has a factor $z^2+f(y)z+g(y)$ where $f(y),\,g(y) \not\equiv 0$, then it is easy to see that $z^2-f(y)z+g(y)$ is also a factor of $T(y,z)$. In this case, we have
\begin{eqnarray*}
z^4+h_2(y)z^2+h_0(y) &=&(z^2+f(y)z+g(y))(z^2-f(y)z+g(y))\\
    &=&z^4+(2g(y)-f^2(y))z^2+g^2(y),
\end{eqnarray*}
Hence $g^2(y)=t_0(y)$, which is impossible since $t_0(y) \in \BC[y]$ is a polynomial of positive degree and it does not have any repeated factors, by Lemma \ref{E}. 

Assume that $T(y,z)$ is reducible. Then by the above arguments, we may suppose that $T(y,z)=(z^2+g(y))(z^2+t_2(y)-g(y))$
where $g(y) \not\equiv 0$. In this case $g(y)(t_2(y)-g(y))=t_0(y)$ which implies that $\frac{t_2(y)^2}{4}-t_0(y)=(g(y)-\frac{t_2(y)}{2})^2$. Note that $t_2(y)^2-4t_0(y)=(4+S_{n-2}(y)^2)(S_{n-2}(y)-S_{n-3}(y))^2$. Hence $$4+S_{n-2}(y)^2=(2g(y)-t_2(y))^2/(S_{n-2}(y)-S_{n-3}(y))^2.$$ Let $h(y)=(2g(y)-t_2(y))/(S_{n-2}(y)-S_{n-3}(y)) \in \BC[y]$ then $4+S_{n-2}(y)^2=h(y)^2$, i.e. $(h(y)-S_{n-2}(y))(h(y)+S_{n-2}(y))=4$. It follows that both $h(y)-S_{n-2}(y)$ and $h(y)+S_{n-2}(y)$ are constant polynomials, and so is $S_{n-2}(y)$. This can not occur since $n \not= 0,\,1,\,2$. Therefore $T(y,z)$ is irreducible in $\BC[y,z]$.
\end{proof}

We now complete the proof of Theorem \ref{character}. Note that $T(2,z)=z^4-(n+1)z^2+n-1$ is not divisible by $(-2 + 3 y - y^3 + z^2) \mid_{y=2}=z^2-4$, since $T(2,\pm 2)=11-3n \not=0$. Hence $T(y,z)$ is not divisible by $-2 + 3 y - y^3 + z^2$ and so, by Proposition \ref{irred}, $\alpha^2 Q'=(-2 + 3 y - y^3 + z^2)T(y,z)$ has exactly 2 irreducible factors. Therefore $V$ has exactly 2 irreducible components.

\subsection{The case $\gcd (2n+1,3)=3$} From the proof of Lemma \ref{z=0} and the proof of Theorem \ref{character} for the case $\gcd(2n+1,3)=1$, it is easy to see that in this case $V$ has exactly 3 irreducible components, where one of them is $\{z=0, \, y=1\}$.

\end{document}